\documentclass[12pt]{article}
\usepackage[T1]{fontenc}
\usepackage{fouriernc}

\usepackage{graphicx}
\usepackage{amsmath}
\usepackage{amsthm}
\usepackage{amsfonts}
\usepackage{amssymb}
\usepackage{fullpage}
\usepackage{url}

\usepackage{pgf}
\usepackage{tikz}
\usetikzlibrary{arrows,automata}

\newtheorem{theorem}{Theorem}
\newtheorem{proposition}[theorem]{Proposition}
\newtheorem{corollary}[theorem]{Corollary}

\newtheorem{lemma}[theorem]{Lemma}

\DeclareMathOperator{\sub}{sub}
\DeclareMathOperator{\inv}{inv}

\title{Some properties of a Rudin--Shapiro-like sequence}

\author{Philip Lafrance, Narad Rampersad, Randy Yee\footnote{The
    second author was supported by an NSERC Discovery Grant.  The
    first and third authors were supported by NSERC USRAs.}\\
Department of Mathematics and Statistics, University of Winnipeg\\
515 Portage Ave., Winnipeg, Manitoba, R3B 2E9 Canada\\
\texttt{philip\_lafrance@hotmail.com}, \texttt{narad.rampersad@gmail.com},\\
\texttt{r7yee@uwaterloo.ca}}

\begin{document}
\maketitle
\begin{abstract}
  We introduce the sequence $(i_n)_{n \geq 0}$ defined by $i_n =
  (-1)^{\inv_2(n)}$, where $\inv_2(n)$ denotes the number of inversions
  (i.e., occurrences of $10$ as a scattered subsequence) in the binary
  representation of $n$.  We show that this sequence has many
  similarities to the classical Rudin--Shapiro sequence.  In
  particular, if $S(N)$ denotes the $N$-th partial sum of the sequence
  $(i_n)_{n \geq 0}$, we show that $S(N) = G(\log_4 N)\sqrt{N}$, where
  $G$ is a certain function that occillates periodically between
  $\sqrt{3}/3$ and $\sqrt{2}$.
\end{abstract}

\section{Introduction}
Loosely speaking, a \emph{digital sequence} is a sequence whose $n$-th
term is defined based on some property of the digits of $n$ when
written in some chosen base.  The prototypical digital sequence is the
\emph{sum-of-digits function} $s_k(n)$, which is equal to the sum of
the digits of the base-$k$ representation of $n$.  Of course, when
$k=2$, the sequence $s_2(n)$ counts the number of $1$'s in the binary
representation of $n$.  By considering only the parity of $s_2(n)$,
one obtains the classical \emph{Thue--Morse sequence} $(t_n)_{n \geq
  0}$, defined by $t_n = (-1)^{s_2(n)}$.  That is,
\[
\begin{array}{cccccccccc}
(t_n)_{n \geq 0} = & +1 & -1 & -1 & +1 & -1 & +1 & +1 & -1 & \cdots
\end{array}
\]

Similarly, if one denotes by $e_{2;11}(n)$ the number of occurrences
of $11$ in the binary representation of $n$, one obtains the
\emph{Rudin--Shapiro sequence} $(r_n)_{n \geq 0}$ by defining $r_n =
(-1)^{e_{2;11}(n)}$.  That is,
\[
\begin{array}{cccccccccc}
(r_n)_{n \geq 0} = & +1 & +1 & +1 & -1 & +1 & +1 & -1 & +1 & \cdots
\end{array}
\]

Traditionally, digital sequences have been defined in terms of the
number of occurrences of a given block in the digital
representation of $n$.  Here we define a sequence based on the number
of occurrences of certain patterns as \emph{scattered subsequences} in
the digital representation of $n$.

Let $a_0 a_1 \cdots a_\ell$ be the base-$k$ representation of an
integer $n$; that is
\[
n = \sum_{j=0}^\ell a_j k^{\ell-j}, \quad\quad a_j \in \{0,1,\cdots, k-1\}.
\]
A \emph{scattered subsequence} of $a_0 a_1 \cdots a_\ell$ is a word
$a_{j_1} a_{j_2} \cdots a_{j_t}$ for some collection of indices $0
\leq j_1 < j_2 < \cdots < j_t \leq \ell$.  Let $p$ be any word over
$\{0,1,\ldots,k-1\}$.  We denote the number of occurrences of $p$ as a
scattered subsequence of the base-$k$ representation of $n$ by
$\sub_{k;p}(n)$.  In particular, $\sub_{2;10}(n)$ denotes the number
of ocurrences of $10$ as a scattered subsequence of the binary
representation of $n$.  For example, since the binary representation
of the integer $12$ is $1100_2$ and the word $1100$ has four
occurrences of $10$ as a subsequence, we have $\sub_{2;10}(12) = 4$.

The quantity $\sub_{2;10}(n)$ can be viewed alternatively as the
number of \emph{inversions} in the binary representation of $n$.  In
general, over an alphabet $\{0,1,\ldots,k-1\}$, an \emph{inversion} in
a word $w$ is an occurrence of $ba$ as a scattered subsequence of $w$,
where $a,b \in \{0,1,\ldots,k-1\}$ and $b>a$.  For this reason, in the
remainder of this paper we will write $\inv_2(n)$ to denote
$\sub_{2;10}(n)$.

We now define the sequence $(i_n)_{n \geq 0}$ by $i_n =
(-1)^{\inv_2(n)}$.  That is,
\[
\begin{array}{cccccccccc}
(i_n)_{n \geq 0} = & +1 & +1 & -1 & +1 & +1 & -1 & +1 & +1 & \cdots
\end{array}
\]
We will show that this sequence has many similarities with the
Rudin--Shapiro sequence.

When studying digital sequences, one often looks at the
\emph{summatory function} of the sequence to get a better idea of the
long-term behaviour of the sequence.  For instance, Newman
\cite{New69} and Coquet \cite{Coq83} studied the summatory function of
the Thue--Morse sequence taken at multiples of $3$.  In particular,
\[
\sum_{0 \leq n < N} t_{3n} = N^{\log_4 3} G_0(\log_4 N) +
\frac13\eta(N),
\]
where $G_0$ is a bounded, continuous, nowhere
differentiable, periodic function with period $1$, and
\[
\eta(N)=
\begin{cases}
0 & \text{if $N$ is even,}\\
(-1)^{3N-3} & \text{if $N$ is odd.}\\
\end{cases}
\]

Similarly, Brillhart, Erd\H{o}s, and Morton \cite{BEM83}, and
subsequently, Dumont and Thomas \cite{DT89} studied the summatory
function of the Rudin--Shapiro sequence.  In this case,
\[
\sum_{0 \leq n < N} r_n =\sqrt{N} G_1(\log_4 N)
\]
where again $G_1$ is a bounded, continuous, nowhere
differentiable, periodic function with period $1$.  We will show that
the summatory function of the sequence $(i_n)_{n \geq 0}$ has the same
form as that of the Rudin--Shapiro sequence.

For more on digital sequences, the reader may consult
\cite[Chapter~3]{AS03}, as well as \cite{DG10}.  Brillhart and Morton
\cite{BM96} have also given a nice expository account of their work on
the Rudin--Shapiro sequence.

\section{Alternative definitions of the sequence $(i_n)_{n \geq
    0}$}\label{definitions}

Let us begin by recalling the definition of $(i_n)_{n \geq 0}$: we
have $i_n = (-1)^{\inv_2(n)}$, where $\inv_2(n)$ denotes the number of
ocurrences of $10$ as a scattered subsequence of the binary
representation of $n$.

Our first observation is that $(i_n)_{n \geq 0}$ is a
\emph{$2$-automatic sequence} (in the sense of Allouche and Shallit
\cite{AS03}).  It is generated by the automaton pictured in
Figure~\ref{DFAO}.  (We do not recapitulate the definitions of
\emph{automatic sequence} or \emph{automaton} here: the reader is
referred to \cite{AS03}.)

\begin{figure}[h!]
\centering
\begin{tikzpicture}[->,>=stealth',shorten >=1pt,auto,node distance=2.8cm,
                    semithick]

  \node[initial,state] (A)              {${+1 \choose +1}$};
  \node[state]         (B) [right of=A] {${-1 \choose +1}$};
  \node[state]         (C) [right of=B] {${-1 \choose -1}$};
  \node[state]         (D) [right of=C] {${+1 \choose -1}$};

  \path (A) edge [loop above] node {0} (A)
            edge [bend left]  node {1} (B)
        (B) edge [bend left]  node {0} (C)
            edge [bend left]  node {1} (A)
        (C) edge [bend left]  node {0} (B)
            edge [bend left]  node {1} (D)
        (D) edge [loop above] node {0} (D)
            edge [bend left]  node {1} (C);
\end{tikzpicture}
\caption{Automaton generating the sequence $(i_n)_{n \geq 0}$}\label{DFAO}
\end{figure}
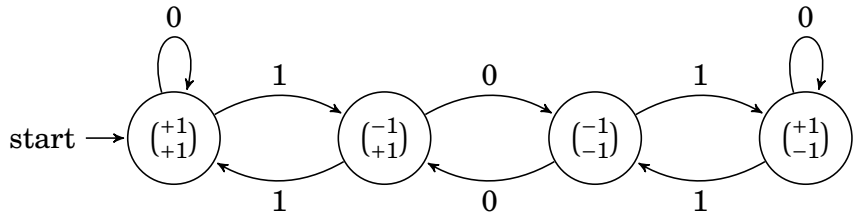

The automaton calculates $i_n$ as follows: the binary digits of $n$
are processed from most significant to least significant, and when the
last digit is read, the automaton halts in the state
\[
{ (-1)^{s_2(n)} \choose (-1)^{\inv_2(n)}}.
\]
In particular, $i_n$ is given by the lower component of the label of
the state reached after reading the binary representation of $n$ (the
first component has the value $t_n$).

Consequently, $(i_n)_{n \geq 0}$ can be generated by iterating the
morphism $g : \{A,B,C,D\}^* \to \{A,B,C,D\}^*$ defined by
\[
A \to  AB, \quad B \to CA, \quad C \to BD, \quad D \to DC,
\]
to obtain the infinite sequence
\[
ABCABDABCADCABCA \cdots
\]
and then applying the recoding
\[
A, B \to {+}1, \quad C, D \to {-}1.
\]
(The reader may again consult \cite[Chapter~6]{AS03} for the standard
conversion between automata and morphisms.)  Compare this to the
Rudin--Shapiro sequence, which is obtained by iterating
\[
A \to  AB, \quad B \to AC, \quad C \to DB, \quad D \to DC,
\]
and then applying the same recoding as above.

The sequence $(i_n)_{n \geq 0}$ also satisfies certain recurrence
relations.  To begin with, we have
\begin{eqnarray}
i_{2n} & = & i_n t_n \label{i_even} \\
i_{2n+1} & = & i_n, \label{i_odd}
\end{eqnarray}
where $t_n$ is the $n$-th term of the Thue--Morse sequence, as defined in
the introduction.  To see this, note that if $w$ is the binary
representation of $n$, then $w0$ is the binary representation of $2n$.
The number of occurences of $10$ as a subsequence of $w0$ equals the
number of occurrences of $10$ as a subsequence of $w$ plus the number
of $1$'s in $w$.  Thus
\[
i_{2n} = (-1)^{\inv_2(2n)} = (-1)^{\inv_2(n)+s_2(n)} =
(-1)^{\inv_2(n)}(-1)^{s_2(n)} = i_n t_n.
\]
Now the binary representation of $2n+1$ is $w1$, and appending the $1$
to $w$ creates no new occurrences of $10$, so $i_{2n+1} = i_n$.

\begin{proposition}\label{relations}
The sequence $(i_n)_{n \geq 0}$ satisfies the following recurrence
relations:
\begin{eqnarray*}
i_{4n} & = & i_n \\
i_{4n+1} & = & i_{2n} \\
i_{4n+2} & = & -i_{2n} \\
i_{4n+3} & = & i_n.
\end{eqnarray*}
\end{proposition}

\begin{proof}
First, recall that the Thue--Morse sequence satisfies the relations
\[
t_{2n} = t_n \quad\text{and}\quad t_{2n+1} = -t_n.
\]
Now we have
\[
i_{4n} = i_{2n} t_{2n} = i_{2n} t_n = i_n t_n t_n = i_n,
\]
where we have applied \eqref{i_even} twice.  Similarly, we get
\[
i_{4n+1} = i_{2(2n)+1} = i_{2n+1} = i_n
\]
by applying \eqref{i_odd} twice.  Next, we calculate
\[
i_{4n+2} = i_{2(2n+1)} = i_{2n+1} t_{2n+1} = i_n (-t_n) = -i_{2n},
\]
and finally,
\[
i_{4n+3} = i_{2(2n+1)+1} = i_{2n+1} = i_n.
\]
\end{proof}

The relations of Proposition~\ref{relations} can be represented in
matrix form as follows.  Define the matrices
\[
\Gamma_0 = \left( \begin{array}{cc}1&0\\0&1\end{array} \right) =
\Gamma_3, \quad
\Gamma_1 = \left( \begin{array}{cc}0&1\\-1&0\end{array} \right), \quad
\Gamma_2 = \left( \begin{array}{cc}0&-1\\1&0\end{array} \right).
\]
For $n=0,1,2,\ldots$ define
\[
V_n = {i_n \choose i_{2n}}.
\]
Then for $n=0,1,2,\ldots$ and $r = 0,1,2,3$, we have
\begin{equation}\label{gamma_rec}
V_{4n+r} = \Gamma_r V_n.
\end{equation}

\section{The summatory function}

Define the \emph{summatory function} $S(N)$ of $(i_n)_{n \geq 0}$ as
\[
S(N) = \sum_{0 \leq n \leq N} i_n.
\]
The first few values of $S(N)$ are:
\begin{center}
\begin{tabular}{|c|cccccccc|}
\hline
$N$ & 0 & 1 & 2 & 3 & 4 & 5 & 6 & 7 \\ \hline
$S(N)$ & 1 & 2 & 1 & 2 & 3 & 2 & 3 & 4 \\
\hline
\end{tabular}
\end{center}
The graph given in Figure~\ref{summatory} is a plot of the function
$S(N)$.  The upper and lower smooth curves are plots of the functions
$\sqrt{2}\sqrt{N}$ and $(\sqrt{3}/3)\sqrt{N}$.

\begin{figure}[h]
\centering
\includegraphics[scale=0.8]{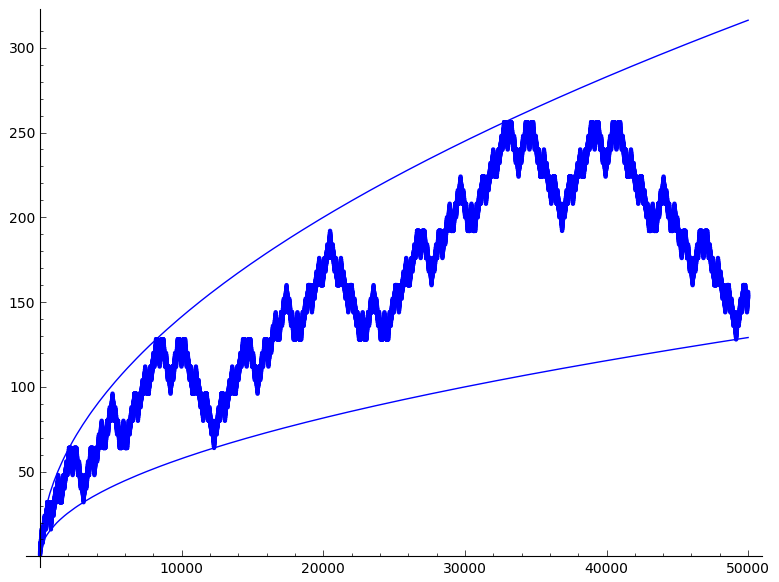}
\caption{A plot of the function $S(N)$}\label{summatory}
\end{figure}

\begin{theorem}\label{sum_growth}
There exists a bounded, continuous, nowhere differentiable,
periodic function $G$ with period $1$ such that
\[
S(N) = \sqrt{N}G(\log_4 N).
\]
\end{theorem}

A plot of the function $G$ is given in Figure~\ref{periodicG}.

\begin{figure}[h]
\centering
\includegraphics[scale=0.8]{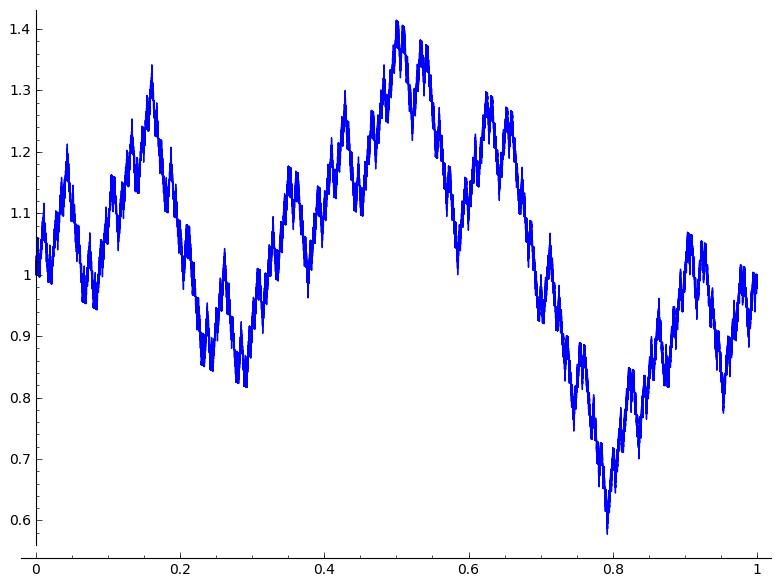}
\caption{A plot of the periodic function $G$}\label{periodicG}
\end{figure}

 The proof of Theorem~\ref{sum_growth} is a straightforward application
of the following result \cite[Theorem~3.5.1]{AS03} (stated here in
slightly less generality):

\begin{theorem}\label{thm3_5_1}
Let $k \geq 2$ be an integer.  Suppose there exist an integer $d \geq
1$, a sequence of vectors $(V_n)_{n \geq 0}$, $V_n \in \mathbb{C}^d$,
and $k$ $d \times d$ matrices $\Gamma_0, \ldots, \Gamma_k$ such that
\begin{enumerate}
\item $V_{kn+r} = \Gamma_rV_n$ for $n=0,1,2,\ldots$ and $r =
  0,1,\ldots,k-1$;
\item $\|V_n\| = O(\log n)$;
\item $\Gamma := \Gamma_0 + \cdots + \Gamma_k = cI$, where $I$ is the
  $d \times d$ identity matrix and $c>0$ is some constant.
\end{enumerate}
Then there exists a continuous function $F : \mathbb{R} \to
\mathbb{C}^d$ of period $1$ such that if $A(N) = \sum_{0 \leq n \leq N}
V_n$, then
\[
A(N) = N^{\log_k c} F(\log_k N).
\]
\end{theorem}

Theorem~\ref{sum_growth} (except for the non-differentiability of $G$)
now follows from Theorem~\ref{thm3_5_1} by taking $k=4$, $d=2$, and
letting the $\Gamma_r$ and $V_n$ be as defined in Section~\ref{definitions}.
Condition~(1) is Eq.~\eqref{gamma_rec}; Condition~(2) is clear, since
$i_n \in \{-1,+1\}$; Condition~(3) holds with $c=2$.  Now $S(N)$ is
the first component of the vector $A(N)$; if we take $G$ to be the
function obtained by projecting $F$ onto its first component,
Theorem~\ref{thm3_5_1} gives
\[
S(N) = N^{\log_4 2} G(\log_4 N) = N^{1/2} G(\log_4 N),
\]
as required.  All the assertions of Theorem~\ref{sum_growth} have now
been established, except for the nowhere differentiability of $G$.  To
obtain this, we note that the proof of \cite{Ten97} for the
summatory function of the Rudin--Shapiro sequence goes through here
for $S(N)$ without modification.

\begin{proposition}\label{sumIdentities}
The function $S(n)$ satisfies the following recurrence relations:
\begin{eqnarray}
S(4n) &=& 2S(n) - i_n\\
S(4n+1) &=& 2S(n) - i_n + i_{2n}\\
S(4n+2) &=& 2S(n) - i_n\\
S(4n+3) &=& 2S(n).
\end{eqnarray}
\end{proposition}

\begin{proof}
Let $A(n) = \sum_{0 \leq j \leq n} V_j$ (as in Theorem~\ref{thm3_5_1}).
Then
\begin{align*}
A(4n+3) &= \sum_{0 \leq j \leq 4n+3} V_j \\
&= \sum_{0 \leq r < 4} \sum_{0 \leq j \leq n} V_{4j+r} \\
&= \sum_{0 \leq r < 4} \sum_{0 \leq j \leq n} \Gamma_r V_j
\quad\quad\text{(by \eqref{gamma_rec})}\\
&= \sum_{0 \leq r < 4} \Gamma_r \sum_{0 \leq j \leq n} V_j \\
&= \left( \sum_{0 \leq r < 4} \Gamma_r \right) \left( \sum_{0 \leq j \leq
    n} V_j \right) \\
&= 2I  \sum_{0 \leq j \leq n} V_j \\
&= 2A(n).
\end{align*}
Now $S(n)$ is the first component of $A(n)$, so we have $S(4n+3) =
2S(n)$.  We thus have
\[
S(4n+r) = S(4n+3) - \sum_{r < \ell \leq 3} i_{4n+\ell} = 2S(n) -
\sum_{r < \ell \leq 3} i_{4n+\ell}.
\]
Applying the relations of Proposition~\ref{relations} now gives the
claimed relations for $S(n)$.
\end{proof}

\begin{corollary}\label{parity}
Let $n$ be a positive integer. Then $S(n)$ and $n$ have opposite parity.
\end{corollary}

\begin{corollary}
Let $n$ be a positive integer. Then $$\frac{S(n) - 2}{2} \leq S\Bigg(\Big\lfloor \frac{n}{4} \Big\rfloor\Bigg) \leq \frac{S(n) + 2}{2}$$
\end{corollary}

Next we identify the positions of certain local maxima and minima
of $S(n)$.  For a positive integer $k$ define the interval: $I_k =
[2^{2k -1}, 2^{2k+1} - 1]$.

\begin{theorem}\label{I_k_max}
For all $k\geq1$, if $n\in I_k$, then $S(n) \leq 2^{k+1}$. Moreover, $S(n) = 2^{k+1}$ only when $n = 2^{2k+1}-1$.
\end{theorem}

\begin{proof}
We proceed by induction on $k$.
The result clearly holds for $k=1$, so suppose the result holds for some $k\geq 1$ and consider $n\in I_{k+1} = [2^{2(k+1)-1},2^{2(k+1) +1}-1] =  [2^{2k+1}, 2^{2k+3} - 1]$. It will be useful for us to write $n = 4m+d$ for some positive integer $m$ and $d\in \{0,1,2,3\}$. Further, we make the observation that $m\in I_k$ for any $n$ in $I_{k+1}$.
\medskip

\noindent\underline{Case 1}: $m \neq 2^{2k+1}-1$. \\
By the induction hypothesis, $S(m) \leq 2^{k+1}-1$. Thus
\begin{eqnarray*}
S(n) = S(4m+d) &\leq& 2S(m)+2 \\
                          &\leq& 2(2^{k+1}-1) +2 \\
                         &=& 2^{k+2}.
\end{eqnarray*}

\noindent\underline{Case 2}: $m = 2^{2k+1}-1$. \\
Again by the induction hypothesis, $S(m) = 2^{k+1}$. We have 4 subcases: \\
\underline{$n = 4m+3$}: By Proposition \ref{sumIdentities}, $S(n) = 2S(m) = 2^{k+2}$.
\\
\underline{$n = 4m+2$}:
Then $n = 2^{2(k+1)+1} - 2$. We make the observation that in base 2, $n+1$ consists  only of $2(k+1)+1$ ones. Hence, $\inv_2(n+1) = 0$ and so $i_{n+1} = 1$. This yields: $$S(n) = S(n+1) - i_{n+1} = 2^{k+2} - 1 \leq 2^{k+2}$$ by the above subcase.  \\
\underline{$n = 4m + 0$}: Here $n = 2^{2(k+1)+1} - 4$. Observe that the base 2 representation of $m$ consists of exactly $2k+1$ ones, and hence $i_{m} = 1$ (since $m$ will have no inversions). Thus $$S(4m) = 2S(m) - i_m = 2^{k+2} -1 \leq 2^{k+2}. $$
\underline{$n = 4m+1$}: Here, $n$ may be expressed as $n = 2^{2(k+1)+1} - 3$. We claim $n$ has an odd number of inversions since its binary representation consists of $2(k+1)-1$ ones followed by `01'. It follows that $i_{4m+1} = -1$, giving $$S(4m+1) = S(4m) +i_{4m+1} = S(4m) - 1 \leq 2^{k+2} - 2.$$
It should also be noted that using induction and the above identities, $$S(2^{2k+3}-1) = 2S(4(2^{2k+1}-1) +3) = 2S(2^{2k+1}-1) = 2^{k+2}$$ for all $k \geq 1$.

It remains to show that the only position at which $S(n) = 2^{k+1}$ is $n= 2^{2k+1}-1$.  Let $n\in I_{k+1} = [2^{2k+1}, 2^{2k+3} - 1]$ and suppose that $S(n) = 2^{k+2}$. Since $S(n)$ is even, $n$ must be odd. Then either $n = 4m+1$ or $n = 4m+3$ for some integer $m$. Suppose the former. Then,
$$S(n) = S(4m+1) = 2S(m) - i_{m} + i_{2m} = 2S(m) - i_{m} + i_mt_m = 2S(m) +i_m(t_m - 1).$$
Obviously $t_m = \pm 1$. Suppose that $t_m = +1$. Then we get that
$S(n) = 2S(m) = 2^{k+2}$ and thus, $S(m) = 2^{k+1}$. Now by the
induction hypothesis, $m = 2^{2k+1} - 1$. Then in base $2$, $m =
111\dotsm1$ ($2k+1$ ones), contradicting the fact that $t_m = +1$. So then it must be that indeed $t_m = -1$. Moreover, if $i_{m}(t_m-1) = -2$, then $S(m) = 2^{k+1}+1$, a contradiction, since $m \in I_k$. Hence $i_{m} = -1$ and it follows that $$S(n) = 2S(m) +2 = 2^{k+2},$$ which implies that $S(m) = 2^{k+1} - 1$.

Observe that $$S(m-1) = S(m) - i_{m}  = (2^{k+1} - 1) +1 = 2^{k+1}.$$ Consequently, $S(m-1)$ achieves the maximum for $I_k$ and so $m-1$ is the endpoint for the interval $I_k$. This yields that $m$ is in fact the first element in $I_{k+1}$. In other words, $m = 2^{2k+1}$, contradicting the fact that $m \in I_k$. Thus we finally conclude that $n \neq 4m+1$.

We now claim that $m = 2^{2k+1}-1$. Suppose that it isn't. Then by the induction hypothesis, $S(m) \leq 2^{k+1}-1$. By the above argument, $n = 4m+3$, so we have
\begin{equation*}
    2^{k+2} = S(n) = 2S(m) \leq 2(2^{k+1}-1) = 2^{k+2}-2 < 2^{k+2}
  \end{equation*}
  which is a contradiction. Hence the only possible choice of $n$ is $n = 4(2^{2k+1}-1)+3 = 2^{2k+3}-1 =2^{2(k+1)+1}-1.$ We have already seen that $S(n)$ is indeed $2^{k+2}$, so this completes the proof.
\end{proof}

\begin{corollary}\label{maxLimit} $\lim\limits_{k \rightarrow \infty} \frac{S(2^{2k+1}-1)}{\sqrt{2^{2k+1}-1}} = \sqrt{2}$.
\end{corollary}

\begin{theorem}\label{I_k_min}
For $k \geq 1$ and $n \in I_k$, $S(n) \geq 2^{k-1}$. Moreover, $S(n) = 2^{k-1}$ if and only if $n = 3 \cdot 4^{k-1} -1.$
\end{theorem}
\begin{proof} This theorem is true for $k=1$, so assume the result for
  an arbitrary $k \geq 1$ and consider $I_{k+1}=
  [2^{2(k+1)-1},2^{2(k+1)+1}-1]$. Let $n$ be in $I_{k+1}$. As before,
  we will let $n=4m+d$,  where $d \in \{0,1,2,3\}$.  Note that $m \in I_k$.
  We consider 2 cases.
\medskip

\noindent \underline{Case 1}: $m = 3 \cdot 4^{k-1} -1$. \\
\noindent Then $S(m) = 2^{k-1}$. The possibilities for $n$ are $n_d =
4(3\cdot4^{k-1} -1)+d = 3 \cdot 4^{k} -(4-d)$, for $d \in
\{0,1,2,3\}$. Now observe that $n_3 = 3 \cdot 4^k -1$ expressed in
binary has the form $$10 \underbrace{11\cdots 1}_{\text{$2k$ `1's}}.$$
It follows that $i_{n_3} = -1$. By observing the binary expansions of
$n_2, n_1, n_0$, we can determine that $i_{n_2} = -1, i_{n_1} = +1$ and
$i_{n_0} = -1$. By Proposition \ref{sumIdentities}, $S(n_3) =
2^k$. Working backwards from $n_3$, it can be seen that $S(n_d) > 2^k$
for $d = 0,1,2$. Hence in this case, the only position in which $S(n)
= 2^{k}$ is $n = 3 \cdot 4^{k} -1$.
\medskip

\noindent  \underline{Case 2}: $m \neq 3 \cdot 4^{k-1} -1$. \\
In this case, $S(n) \geq 2S(m) - 2 \geq 2(2^{k-1}+1) -2 = 2^{k}$, which is all we need.

Having now established the lower bound, we now only need to show that it is unique. Assume $S(n) = 2^k$. This implies that $n$ is odd, so we begin by supposing $n = 4m +1$. Then
$$S(n) = S(4m+1) = 2S(m) - i_{m} + i_{2m} = 2S(m) - i_{m} + i_mt_m = 2S(m) +i_m(t_m - 1).$$
In a fashion similar to that seen in the upper bound, we find that
$i_m = +1$ and $m \neq 3 \cdot 4^{k-1}-1$. Hence $S(m-1) = S(m) -
i_{m} = 2^{k-1}+1-1 = 2^{k-1}$. By the induction hypothesis, there is
only one value in $I_k$ such that $S(m_0)=2^{k-1}$. Namely $m_0 = 3
\cdot 4^{k-1}-1$. Hence $m = m_0 +1 = 3 \cdot 4^{k-1}$. We may thus
conclude that the only possibility for $n$ in this case is $n = 4(3
\cdot 4^{k-1}) + 1$. Under examination of the binary representation of
$n$ and $n-1$ as well as the fact that $n-2 = 4m_0+3$ and consequently
$S(n-2) = 2S(m_0)=2^{k}$, we find that this is not the case. Hence $n$ has the form $4m+3$.

If $m \neq 3 \cdot 4^{k-1} -1$, then $S(m) \geq 2^{k-1} +1$. By Proposition \ref{sumIdentities}, $S(n) \geq 2(2^{k-1}+1) > 2^{k}$, contradicting the assumption that $S(n) = 2^k$. It follows that $n = 4m+3 = 4(3 \cdot 4^{k-1} -1) +3 = 3 \cdot 4^{k} -1$ is the only possibility. As we have already verified that $S(n)$ does indeed equal $2^{k}$ for this value of $n$, we have a unique minimum for $S(n)$ on $I_{k+1}$. The result now follows.
\end{proof}

Theorems~\ref{I_k_min} and \ref{I_k_max} show that
\[
\liminf_{n \to \infty} \frac{S(n)}{\sqrt{n}} \leq \frac{\sqrt{3}}{3}
\quad\text{and}\quad
\limsup_{n \to \infty} \frac{S(n)}{\sqrt{n}} \geq \sqrt{2},
\]
respectively.  In the next section, we will show that the lower and
upper limits are in fact equal to $\sqrt{3}/3$ and $\sqrt{2}$.  That
is, we will prove

\begin{theorem}\label{lim_inf_sup}
We have
\[
\liminf_{n \to \infty} \frac{S(n)}{\sqrt{n}} = \frac{\sqrt{3}}{3}
\quad\text{and}\quad
\limsup_{n \to \infty} \frac{S(n)}{\sqrt{n}} = \sqrt{2}.
\]
\end{theorem}

\section{Establishing the upper and lower limits of $S(n)/\sqrt{n}$}
The following lemma provides us with some tools to work with for the proof of the upper limit of $S(n)/\sqrt{n}$.

\begin{lemma}\label{dividingI}
\begin{align}  & S(n + 2^{2k}) = -S(n) + 3(2^{k}),  & 2^{2k} \leq n \leq 2^{2k+1} -1, k \geq 1; \label{eqq3}\\
               & S(n + 3 \cdot 2^{2k}) = S(n) + 2^{k},  & 0 \leq n \leq 2^{2k} -1, k \geq 1;  \label{eqq4}\\
               & S(n + 2^{2k+1}) = -S(n) + 2^{k+2}\label{eqq2},     & 2^{2k+1} \leq n \leq 2^{2k+2} -1, k \geq 1; \\
               & S(n + 3 \cdot 2^{2k+1}) = S(n) + 2^{k+1},  & 0 \leq n \leq 2^{2k+1} -1, k \geq 1. \label{eqq1}
\end{align}
\end{lemma}

\begin{proof} Consider equation (\ref{eqq2}). We will show
  that for an arbitrary $k \geq 1$, $S(2^{2k+2}) + S(2^{2k+1})=
  2^{k+2}$, so that by rearranging we obtain equation (\ref{eqq2}) with $n = 2^{2k+1}$. Observing the binary representations and using Theorem~\ref{I_k_max}, we find that $S(2^{2k+1}) = 2^{k+1}-1$. So we must show that $S(2^{2k+2}) = 2^{k+1}+1$, or equivalently (again using the binary representation), that $S(2^{2k+2}-1) = 2^{k+1}$. It may be verified that this is true for $k=1$, so we proceed via induction on $k$. Assuming the result holds for $k$, we consider the $k+1$ case. $$S(2^{2(k+1)+2}-1) = S(4(2^{2k+2}-1)+3) = 2S(2^{2k+2}-1) = 2(2^{k+1}) = 2^{k+2},$$ hence the result is true for all $k \geq 1$.

Now, for $2^{2k+1} \leq j \leq 2^{2k+2} -1$, we claim that it must be the case that $i_j = -i_{j + 2^{2k+1}}$. This is because the difference in the inversion counts of $j$ and $j+2^{2k+1}$ can be attributed solely to those obtained from their respective leading `1's. In fact, the leading `1' of the latter term will give exactly one more inversion than the former. Hence the parity of the inversion counts will be different. It now follows that starting with $n = 2^{2k+1}$ and increasing $n$ successively by one, that $S(n+ 2^{2k+1}) + S(n)= 2^{k+2}$ for each $n$ in the interval $[2^{2k+1},2^{2k+2}-1]$.

We will now prove (\ref{eqq1}). Our first order of business will be to show that for any $k \geq 1$, $S(3 \cdot 2^{2k+1}) = 2^{k+1} + 1$. Considering the binary representation of $3 \cdot 2^{2k+1}$, we find that  $S(3 \cdot 2^{2k+1}) = S(3 \cdot 2^{2k+1}-1)+1$. Therefore it will be sufficient to show that $S(3 \cdot 2^{2k+1}-1) = 2^{k+1}$. For $k=1$ we have $S(23) = 4$, so suppose the result holds for some $k \geq 1$ and consider $k+1$. Since $3 \cdot 2^{2(k+1)+1}-1 = 4(3\cdot 2^{2k+1}-1) +3$, Proposition (\ref{sumIdentities}) gives us that $S(3 \cdot 2^{2(k+1)+1}-1) = 2S(3\cdot 2^{2k+1}-1) = 2^{k+2}$ as desired.

It may be observed from the binary representations that $i_{j + 3 \cdot 2^{2k+1}} = i_j$ for $0 \leq j \leq 2^{2k+1} -1$. This stems from the fact that for each $j$ in this interval, $j + 3 \cdot 2^{2k+1}$ has a different inversion count from $j$ only due to the 2 leading `1's, which can be disregarded when considering the parity of the number of inversions. It now follows that starting with $n = 0$ and increasing $n$ successively by one, that $S(n + 3 \cdot 2^{2k+1}) = S(n) + 2^{k+1}$ for each $n$ in the domain of equation (\ref{eqq1}).

Equations~\eqref{eqq3} and \eqref{eqq4} may be proved in a similar fashion.
\end{proof}

\subsection{Outline of the proof of the upper limit}
In order to prove the upper limit of $\frac{S(n)}{\sqrt{n}}$, our
argument becomes a little bit messy, so we give a brief outline of our
approach:  Recall that $I_k = [2^{2k -1}, 2^{2k+1} - 1]$. Lemma~\ref{dividingI} leads naturally to the following division of $I_{k}\setminus \{2^{2k+1}-1\}$:
\begin{align*} &I_{k,1} = [2^{2k-1}, 3\cdot 2^{2k-2}-1] &&\qquad I_{k,2}=[3\cdot 2^{2k-2}, 2^{2k}-1]\\
   &I_{k,3}= [2^{2k}, 3\cdot 2^{2k-1}-1] &&\qquad I_{k,4} = [3\cdot 2^{2k-1}, 2^{2k+1} -2]. \end{align*}
We attempt to prove that for $n \geq 8$, if $n \neq 2^{2k+1}-1$ for $k \geq 1$, then $\frac{S(n)}{\sqrt{n}} < \sqrt{2}$.
$I_{k,1}$ and $I_{k,2}$ are taken care of by first establishing that the maximum $S(n)$ value on these two intervals is $2^k$, after which the result falls out quite nicely.

 The proof for the interval $I_{k,3}$ demands that we split it up into
 several sub-intervals based on the equations of Lemma~\ref{dividingI}. We show that a local max on $[2^{2k},3\cdot 2^{2k-1}-1]$ occurs at $n_0 = 5\cdot 2^{2(k-1)}$, with $S(n_0) = 3\cdot 2^{k-1}$, which effectively cuts $I_{k,3}$ in half. The algebra comes together for the second half of this division, but the first still requires some work.

Using the formulae once more, we cut this new subinterval into two pieces, $[2^{2k}, 3^2 \cdot 2^{2(k-1) -1}-1]$ and $[3^2 \cdot 2^{2(k-1)-1}, 5\cdot 2^{2(k-1)}-1]$. Again the algebra follows for the latter half, but not the former. We then determine that a max on the former interval occurs at $n = 17 \cdot 2^{2(k-3)}-1$, giving $S(n) \leq 5\cdot 2^{k-2}$, which is strong enough to finally allow us to obtain the desired inequality.

The last interval $I_{k,4}$, with the exception of $n = 2^{2k+1}-1$, ends up being dispatched with relative ease using some simple algebra.

This result along with Theorem \ref{I_k_max} and Corollary \ref{maxLimit} is enough to give the desired result.

\subsection{Establishing the upper limit}
\begin{theorem}\label{sqrt2} Let $n \geq 8$. If $n \neq 2^{2k+1}-1$ for $k \geq 1$, then $\frac{S(n)}{\sqrt{n}} < \sqrt{2} $.
\end{theorem}
In order to prove the above, we must first develop some useful tools. The following few results show that if $n \in [2^{2k-1}, 2^{2k}-1]$, then $S(n) \leq 2^k$ for $k \geq 1$.

\begin{proposition}\label{formProp} Suppose $k \geq 1$. If $m_0$ is of the form \begin{equation}\label{theform} 2^{2k} -1 - \sum\limits_{r=0}^{k-2}\varepsilon_r(3\cdot 2^{2r+1}) - \beta \end{equation} for some combination of $\varepsilon_r\text{{\rm{'s}}} \in\{0,1\}$ and $\beta \in \{0,2\}$, then $4 m_0 + 3$ is also of the above form.
\end{proposition}

\begin{proof} First let $m_0 = 2^{2k} -1 - \sum\limits_{r=0}^{k-2}\varepsilon_r(3\cdot 2^{2r+1})$. Then
\begin{align*} 4 m_0 +3 &= 4\left( 2^{2k} -1 - \sum\limits_{r=0}^{k-2}\varepsilon_r(3\cdot 2^{2r+1})\right) +3 \\
                        &= 2^{2(k+1)} -4 - \sum\limits_{r=0}^{k-2}\varepsilon_r(3\cdot 2^{2(r+1)+1}) +3\\
                        &= 2^{2(k+1)} -1 - \sum\limits_{s=1}^{k-1}\varepsilon_{s-1}(3\cdot 2^{2{s}+1}) \qquad (\text{letting } s = r+1).\\
\end{align*}
By letting $\varepsilon_0 = 0$ and re-indexing the $\varepsilon_s$ so that the summands have the form $\varepsilon_{s}(3\cdot 2^{2{s}+1})$, we see that the above is indeed of the desired form. The case where $$m_0 = 2^{2k} -1 - \sum\limits_{r=0}^{k-2}\varepsilon_r(3\cdot 2^{2r+1}) -2$$ is similar, although in this case we will have $\varepsilon_0 = 1$.
\end{proof}

\begin{lemma}
If $n$ may be written in the form seen in equation {\rm{(\ref{theform})}} for some combination of $\varepsilon_r\text{'s} \in\{0,1\}$ and $\beta \in \{0,2\}$, then $i_n = +1$.
\end{lemma}

\begin{proof}
Suppose that $$n = 2^{2k} -1 -
\sum\limits_{r=0}^{k-2}\varepsilon_r(3\cdot 2^{2r+1}) - \beta$$ for
some combination of $\varepsilon_r\text{'s} \in\{0,1\}$ and $\beta \in
\{0,2\}$. We note that the binary form of  $2^{2k}-1$ consists of $2k$
1's. Consider the case when $\beta = 0$. Observe that if we label the
digit positions of the binary representation starting from the right and beginning with 0, subtracting $3 \cdot 2^{2r+1}$ from $2^{2k}-1$, for $0 \leq r \leq k-2$, changes the digits in positions $2r+1$ and $2r+2$ from `1's to `0's. It follows that any $n$ of the above form will have `0's only occurring in blocks of even length. This ensures an even number of inversions, which means $i_n = +1$.

Now let $\beta = 2$. Every $n$ of this form may be obtained subtracting 2 from an $n$ of the form in the above case. Since `0's occur in even blocks, this subtraction will turn the block of zeroes adjacent to the `1' in the 0th position (which could possibly be empty) into `1's and the `1' to the left of the block into a `0'. The changing of the even block of zeroes into `1's will change the inversion number by an even amount, so we only need to check that the new `0' does not create an odd number of inversions. However we also know that excluding the left and rightmost `1's, `1's must come in even blocks as well. The new `0' will thus have an even number of `1's to the left of it (since it turns the right digit in a pair of `1's into a `0'). Hence we still have an even number of inversions, so $i_n = +1$.\end{proof}

\begin{lemma}Given an $n$ of the form in equation (\ref{theform}), $S(n) = 2^k$.
\end{lemma}
\begin{proof}
It may be observed that the result is certainly true for $k=1$, so assume that it is true for an arbitrary $k$ and consider $k+1$. Our approach uses Proposition~\ref{sumIdentities} extensively, so it will be useful to note that the only $\varepsilon_r$ that affects the value of $n$ modulo 4 will be $\varepsilon_0$. Since $\beta$ will also affect this value, it is natural to have 4 cases.

\noindent \underline{Case 1}: $\varepsilon_0 = 1, \beta = 2.$\\
We have:
\begin{align*} n &= 2^{2(k+1)} -1 - \sum\limits_{r=1}^{k-1}\varepsilon_r(3\cdot 2^{2r+1}) - 3 \cdot 2 - 2 \\
                 &= 4\left( 2^{2k} - \sum\limits_{r=1}^{k-1}\varepsilon_r(3\cdot 2^{2(r-1)+1}) \right) - 9\\
                 &= 4\left( 2^{2k} - \sum\limits_{r=0}^{k-2}\varepsilon_{r+1}(3\cdot 2^{2r+1}) \right) - 12 + 3\\
                 &= 4 \left( 2^{2k} - 1 -\sum\limits_{r=0}^{k-2}\varepsilon_{r+1}(3\cdot 2^{2r+1}) -2\right) +3.
\end{align*}
From the induction hypothesis, $$S\left(2^{2k} - 1 - \sum\limits_{r=0}^{k-2}\varepsilon_{r+1}(3\cdot 2^{2r+1}) -2\right) = 2^{k},$$ and so by Proposition~\ref{sumIdentities} $S(n) = 2^k$.\\
\underline{Case 2}: $\varepsilon_0 = 0, \beta = 2.$\\
With a bit of algebra, we find that
\begin{align*} n &= 2^{2(k+1)} -1 - \sum\limits_{r=1}^{k-1}\varepsilon_r(3\cdot 2^{2r+1}) - 2 \\
                 &= 4 \left( 2^{2k} -1 - \sum\limits_{r=0}^{k-2}\varepsilon_{r+1}(3\cdot 2^{2r+1}) \right) +1.
\end{align*}
We thus obtain the following equation for $S(n)$:
\begin{align*}S(n) &= S\left( 4 \left( 2^{2k} -1 -\sum\limits_{r=0}^{k-2}\varepsilon_{r+1}(3\cdot 2^{2r+1}) \right) +1 \right)\\
                       &= 2 S(m) -i_m + i_{2m},
\end{align*}
where $m =2^{2k} - 1 -\sum\limits_{r=0}^{k-2}\varepsilon_{r+1}(3\cdot 2^{2r+1})$.
By observing the binary representation of $m$ and $2m$, we find that $-i_m +i_{2m} = 0$. It follows from the induction hypothesis that $S(n) = 2^{k}$. \\
\underline{Case 3}: $\varepsilon_0 = 1, \beta = 0.$\\
It is not too hard to show that
\begin{align*}
n &= 4\left[\left(2^{2k} -1 - \sum\limits_{r=0}^{k-2}\varepsilon_{r+1}(3\cdot 2^{2r+1})\right) -1\right]+1\\
  &= 4(m-1) +1,
\end{align*}
where $$m = 2^{2k} -1 - \sum\limits_{r=0}^{k-2}\varepsilon_{r+1}(3\cdot 2^{2r+1}) -1.$$ From the induction hypothesis and the fact that $i_m = +1$, we obtain that $S(m-1) = 2^{k}-1$. Hence $S(n) =S(4(m-1)+1) = 2S(m-1) - i_{m-1} +i_{2(m-1)}$.

Now $m$ has an even number of `1's and `0's in its binary representation and ends in `01', so $m-1$ will have an odd number of `1's and `0's and end in `00'. The binary representation of $2(m-1)$ will then have an extra `0' at the end, and since there are an odd number of preceding `1's, the parity of its inversion count will be the opposite of $m-1$, ie. $i_{m-1} = -i_{2(m-1)}$.

From the above lemma, we know that $i_m = +1$. Writing $m$ in the form of (\ref{theform}), we find that its $\beta$ value is $0$. Thus $m-2$ may also be written in the same form, which implies $i_{m-2}=+1$. It follows that $i_{m-1} = -1$, and so $S(n)= 2S(m-1) -(-1) +1 = 2^{k+1} - 2 + 2 = 2^{k+1}$ as needed.

The remaining case is similar to Case 1.

\end{proof}

\begin{theorem}\label{thm_J_k}
Let $J_k = [2^{2k-1}-1, 2^{2k}-1]$. Then for $n\in J_k$, $S(n) = 2^k$ if and only if $n$ is of the form in {\rm{(\ref{theform})}} for some combination of $\varepsilon_r\text{\rm{'s}} \in\{0,1\}$ and $\beta \in \{0,2\}$. Moreover, $2^k$ is the maximum value for the partial sum function over $J_k$.
\end{theorem}

\begin{proof}
We proceed by induction.
Observe that this result holds for $J_1$, so suppose it holds true for some $k \geq 1$ and consider $n \in J_{k+1}$.
Suppose $S(n) = 2^{k+1}$, but $n$ is not of the form in
(\ref{theform}) for any combination of $\varepsilon_r$'s and $\beta$'s. First of all, if $n =4m_0 +3$, then $m_0$ is in $J_{k}$, and $S(n) = 2S(m_0)$, giving $S(m_0) = 2^{k}$. By the induction hypothesis, $m_0$ may be written in the form of equation (\ref{theform}). However it follows from Proposition~\ref{formProp} that $4m_0 +3$ may also be written in same form, contradicting the hypothesis. Thus we may assume $n = 4m_0 +1$. \\

\noindent We note that $S(4m_0 +1) = 2S(m_0) -i_{m_0} - i_{2m_0} = 2^{k+1}.$ Some rearranging gives: $$S(m_0) = \frac{(2^{k+1} + i_{m_0} + i_{2m_0})}{2} = 2^{k} + \frac{( i_{m_0} + i_{2m_0})}{2} \in \{2^{k}-1, 2^{k}, 2^{k} +1\}.$$  It follows from the induction hypothesis that $S(m_0) \neq 2^{k}+1$, so we need only consider the other two cases. Suppose $S(m_0)=  2^{k}$. By the induction hypothesis, $i_{m_0}= +1 $ (else $S(m_0 -1) = 2^{k}+1$), and it is easy to see that $i_{2m_0}=  -1$. Furthermore, the induction hypothesis gives that $m_0$ may be written in the form of equation (\ref{theform}), with $\beta = 2$ (since $\beta = 0$ gives that $4m_0 +1$ is also of the form in (\ref{theform}) with $\beta = 2$, which is a contradiction). From this, we can say that $$4m_0 +1 = 2^{2(k+1)}-1 - \sum\limits_{r=0}^{k-1}\varepsilon_r(3\cdot 2^{2r+1}) - 4,$$ and so $4m_0+3$ can be expressed by an equation of the form (\ref{theform}), implying $i_{4m+3} = +1$. Moreover, $i_{4m_0 +2} = -i_{2m_0} = +1$. This gives that
$$2^{k+1} =2S(m_0) = S(4m_0 + 3) = S(4m_0 +1) + i_{4m_0 +2} + i_{4m_0 +3} = 2^{k+1} +2,$$
which is clearly a contradiction. Thus we must have $S(m_0) = 2^{k} -1$. It now follows that $\frac{( i_{m_0} + i_{2m_0})}{2} = -1$, so $i_{m_0} = i_{2m_0} = -1$. Since $i_{4m_0 +1} = i_{2m_0}$, we have
\begin{eqnarray*}
S(4m_0) &=& S(4m_0 +1)-i_{4m_0+1} \\
              &=& 2S(m_0) - i_{m_0}\\
              &=& 2^{k+1} + 1.
\end{eqnarray*}
Therefore $S(4m_0) -1 = 2^{k+1} = 2S(m_0).$ However $2S(m_0) = 2^{k+1}$, implying that $S(m_0) = 2^{k}$,  a contradiction.

Now we must show that $S(n)\leq 2^k$ for $n\in J_k$. Writing $n = 4m+
d$, where $d \in \{ 0,1,2,3\}$, Proposition~\ref{sumIdentities} tells
us that $S(n) > 2^k$ is only possible if $S(m) = 2^{k-1}$. Hence $m$
can be written in the form of equation (\ref{theform}). If $d = 0$,
then Proposition~\ref{sumIdentities} gives us $S(n) \leq 2^{k}+1$. If
we have equality here, this implies that $S(4(m-1)+3) = 2^k$, and
consequently that $S(m)= S(m-1)$, which is clearly impossible. By
Corollary \ref{parity}, $S(4m) \leq 2^k -1$, which gives us that
$S(4m+1) \leq 2^k$. Finally, since we have that $S(4m+3) = 2^k$ and
$i_{4m+3} = +1$,  it follows that $S(4m+2) \leq 2^k -1$. Therefore no value of
 $d$ allows for $S(n)$ to exceed $2^k$, giving the result.
\end{proof}

\noindent We now finally have the necessary tools to prove Theorem \ref{sqrt2}.

\begin{proof}
We will begin by observing that for $n$ in $I_2 \setminus \{31\} = [8,30]$, $\frac{S(n)}{\sqrt{n}} < \sqrt{2}$. So assume that the statement is true for an arbitrary $k \geq 2$ and consider $I_{k+1}\setminus \{2^{2k+1}-1\}$. We will proceed by breaking up this interval into the following 4 pieces:
\begin{align*} &I_{k+1,1} = [2^{2k+1}, 3\cdot 2^{2k}-1] &&\qquad I_{k+1,2}=[3\cdot 2^{2k}, 2^{2k+2}-1]\\
   &I_{k+1,3}= [2^{2k+2}, 3\cdot 2^{2k+1}-1] &&\qquad I_{k+1,4} = [3\cdot 2^{2k+1}, 2^{2k+3} -2]. \end{align*}

\noindent\underline{Case 1}: $n \in I_{k+1,1} \cup I_{k+1,2} = [2^{2k+1}, 2^{2k+2}-1]$.\\
By Theorem~\ref{thm_J_k}, all $S(n)$ values in this range are bounded
above by $2^{k+1}$, so we have
$\frac{S(n)}{\sqrt{n}} \leq \frac{2^{k+1}}{\sqrt{2^{2k+1}}} = \sqrt{2}$. We observe that that equality is possible only when $n = 2^{2k+1}$, but since $S(2^{2k+1}-1) = 2^{k+1}$, we get that $S(2^{2k+1}-1) < 2^{k+1}$, hence the result holds in these two intervals.

\noindent \underline{Case 2:} $n \in I_{k+1,3} =  [2^{2k+2}, 3 \cdot 2^{2k+1}-1]$. \\
Observe that by \eqref{eqq2} $I_{k+1,3}$ is determined entirely by the interval $[2^{2k+1}, 2^{2k+2}-1]$. Since we have that the minimum $S(n)$ value on $I_{k+1}$ occurs at $n = 3 \cdot 2^{2k} -1$, it follows that the maximum on $I_{k+1,3}$ occurs precisely at $n= (3 \cdot 2^{2k} -1) + 2^{2k+1} = 5 \cdot 2^{2k}-1$, with $S(n) = 3 \cdot 2^{k}$. If we consider the interval $[5 \cdot 2^{2k}, 3 \cdot 2^{2k+1}-1]$, we find that
\begin{align*} \frac{S(n)}{\sqrt{n}} \leq \frac{3 \cdot 2^{k}}{\sqrt{5 \cdot 2^{2k}}} = \frac{3}{\sqrt{5}}< \sqrt{2}.
\end{align*}

It remains to show that the bound holds on $[2^{2k+2},5 \cdot 2^{2k}-1]$.
For reasons that will become apparent shortly, it will be convenient
to split this remaining interval into two disjoint pieces,
$[2^{2k+2},9 \cdot 2^{2k-1}-1]$ and $[9 \cdot 2^{2k-1},5\cdot
2^{2k}-1]$. Consider the interval $I_{k,3} = [2^{2k}, 3 \cdot
2^{2k-1}-1]$. We have already established that the unique maximum
value of $S(n)$ on this interval is $3 \cdot 2^{k-1}$, and occurs
exactly at $n_1 = 5 \cdot 2^{2(k-1)} -1$. Using \eqref{eqq3}, we find
that the minimum value on the interval $[2^{2k+1},5 \cdot2^{2k-1}-1]$
occurs at $n_2 = 9 \cdot 2^{2k-2}-1$, with $S(n_2) = 3 \cdot
2^{k-1}$. Finally, we can apply \eqref{eqq2} to obtain that the unique maximum on the interval $[2^{2k+2},9 \cdot 2^{2k-1}-1]$ occurs at $n = 17 \cdot 2^{2(k-2)}-1$, and that $S(n) = 5 \cdot 2^{k-1}$. By some quick algebra, we find that for $n$ in this interval,
$$ \frac{S(n)}{\sqrt{n}} \leq \frac{5 \cdot 2^{k-1}}{\sqrt{2^{2k+2}}} < \sqrt{2}.$$
Lastly we tackle the final piece, $[9 \cdot 2^{2k-1}, 5 \cdot 2^{2k}-1]$. We have that for any $n \in I_{k+1,3}$, $S(n) \leq 3 \cdot 2^{k}$. Thus for any $n$ in this interval
$$ \frac{S(n)}{\sqrt{n}} \leq \frac{3 \cdot 2^{k}}{\sqrt{9 \cdot 2^{2k-1}}} \leq \sqrt{2}.$$
As equality can only hold when $n = 9 \cdot 2^{2k-1}$, we just need to check this value. However, $S(9 \cdot 2^{2k-1}-1) = 3 \cdot 2^{k}$, implying that $S(9 \cdot 2^{2k-1}) \leq 3 \cdot 2^{k}-1$. This gives us a strict inequality and completes the proof for this interval.\\
\underline{Case 3}: $n \in I_{k+1,4} = [3 \cdot 2^{2k+1}, 2^{2k+3}-2]$. \\
Write $n = n'+3 \cdot 2^{2k+1}$.
We observe that \eqref{eqq1} pertains to this interval completely, giving us
\begin{align*}
\frac{S(n)}{\sqrt{n}} &= \frac{S(n' + 3 \cdot 2^{2k+1})}{\sqrt{n'+3
    \cdot 2^{2k+1}}} \\
&= \frac{S(n')}{\sqrt{n'}} \cdot \frac{\sqrt{n'}}{\sqrt{n'+3 \cdot 2^{2k+1}}} +\frac{2^{k+1}}{\sqrt{n'+3 \cdot 2^{2k+1}}}\\
                   & < \frac{\sqrt{2n'} +2^{k+1}}{\sqrt{n'+3\cdot 2^{2k+1}}}.
\end{align*}
Using a little bit of algebra, we find that this is less than $\sqrt{2}$ whenever $0 \leq n' \leq 2^{2k+1}-2$. Since this is within the domain of \eqref{eqq1}, we have the result.

\end{proof}

\subsection{Establishing the lower limit}

\begin{theorem}\label{lower_bound}
 $\frac{S(n)}{\sqrt{n}} > \frac{1}{\sqrt{3}}$ for all $n \geq 1$.
\end{theorem}
\begin{proof}
We can certainly observe this for values of $n$ up to 8, so we can assume that it holds true for all $n$ up to and including $I_j$ where $1 \leq j \leq k$, and consider $I_{k+1}$ for $k \geq 1$.

By Theorem~\ref{I_k_min}, the minimum value of $S$ occurs at $n_0 = 3
\cdot 2^{2k}-1$ with $S(n_0)= 2^k$. It follows quite easily that for
any $n \in I_{k+1,1}=[2^{2k+1},3 \cdot 2^{2k}-1]$, the
inequality $$\frac{S(n)}{\sqrt{n}} \geq \frac{2^k}{\sqrt{3 \cdot
    2^{2k}-1}} > \frac{2^k}{\sqrt{3 \cdot 2^{2k}}} =
\frac{1}{\sqrt{3}}$$ is satisfied.

For $I_{k+1,2}$, observe that \eqref{eqq4} applies exactly to this interval. Hence if $n = n' + 3 \cdot 2^{2k}$, the induction hypothesis gives
\begin{align*} \frac{S(n'+3 \cdot 2^{2k})}{\sqrt{n'+3 \cdot 2^{2k}}} &= \frac{S(n')}{\sqrt{n'}} \cdot \frac{\sqrt{n'}}{\sqrt{n'+3 \cdot 2^{2k}}} + \frac{2^k}{\sqrt{n'+3 \cdot 2^{2k}}}\\
           & > \frac{1}{\sqrt{3}} \cdot \frac{\sqrt{n'}}{\sqrt{n'+3 \cdot 2^{2k}}}+ \frac{2^k}{\sqrt{n'+3 \cdot 2^{2k}}}
\end{align*}
for $n' \geq 1$.
We would like $$\frac{1}{\sqrt{3}} \cdot \frac{\sqrt{n'}}{\sqrt{n'+3 \cdot 2^{2k}}}+ \frac{2^k}{\sqrt{n'+3 \cdot 2^{2k}}} \geq \frac{1}{\sqrt{3}},$$ which with a little bit of work, can be shown to be equivalent to
$$\frac{(\sqrt{n'}+ \sqrt{3}\cdot 2^k)^2}{{n'+3 \cdot 2^{2k}}}  \geq 1.$$ As this is true when $n' \geq 0$ and hence on the domain of the equation, we have the result for $I_{k+1,2} \setminus \{3\cdot 2^{2k}\}$. It is easily verified that $S(3\cdot 2^{2k}) = 2^{k}+1$ and that the bound holds for this value as well, giving us the result for $I_{k+1,2}$.

Now consider $n \in I_{k+1,3}= [2^{2k+2}, 3 \cdot 2^{2k+1}-1]$. Recall
that by Lemma \ref{dividingI}, the values of $S$ on $I_{k+1,3}$ are
completely determined by the values of $S$ on $I_{k+1,1} \cup I_{k+1,2}$. We also know that the
maximum value of $S$ on $I_{k+1,1} \cup I_{k+1,2}$ is $2^{k+1}$. In particular,
$S(n) \geq - 2^{k+1} + 2^{k+2} = 2^{k+1}$.

Finally, let $n \in I_{k+1,4} \cup \{2^{2k+3}-1\}$.  Note that the value of $S$ on
$I_{k+1,4} \cup \{2^{2k+3}-1\}$ is completely determined by the value of $S$ on $[0, 2^{2k+1}-1]$. Moreover, the
minimum value of $S$ on $[0, 2^{2k+1}-1]$ is $1$. By equation
\eqref{eqq1} we obtain that $S(n) \geq 1 + 2^{k+1} > 2^{k+1}$. 

Hence for $n \in I_{k+1,3} \cup I_{k+1,4} \cup \{2^{2k+3}-1\}$, the following inequality holds:
$$\frac{S(n)}{\sqrt{n}} \geq \frac{2^{k+1}}{\sqrt{2^{2k+3}-1}} > \frac{2^{k+1}}{\sqrt{2^{2k+3}}} = \frac{1}{\sqrt{2}} > \frac{1}{\sqrt{3}},$$
thus establishing the result for the remaining piece of $I_{k+1}$ and completing the proof.
\end{proof}

Theorem~\ref{lim_inf_sup} now follows from Corollary \ref{maxLimit} and Theorems~\ref{I_k_max},
\ref{I_k_min}, \ref{sqrt2}, \ref{lower_bound}.

\section{Combinatorial properties}\label{combinat_props}

Both the Thue--Morse sequence and the Rudin--Shapiro sequence have
been extensively studied from the point of view of combinatorics on
words.  Indeed, both of these sequences have many interesting
combinatorial properties.  Before collecting some of the combinatorial
properties of the sequence $(i_n)_{n \geq 0}$, we first recall some
basic definitions.

A word of the form $xx$, where $x$ is non-empty, is called a
\emph{square}.  A \emph{cube} has the form $xxx$, and in general, a
\emph{$k$-power} has the form $xx\cdots x$ ($x$ repeated $k$ times)
and is denoted by $x^k$.  A \emph{palindrome} is word that is equal to
its reversal.  We denote the length of a word $x$ by $|x|$.

\begin{theorem}
The sequence $(i_n)_{n \geq 0}$ contains
\begin{enumerate}
\item no $5$-th powers,
\item cubes $x^3$ exactly when $|x| = 3$,
\item squares $xx$ exactly when $|x| \in \{1,2\} \cup \{3\cdot2^k : k \geq 0\}$.
\item arbitrarily long palindromes.
\end{enumerate}
\end{theorem}

\begin{proof}
First, note that 1) can be deduced from 2) along with a computer
calculation to verify that there are no $5$-th powers of period $3$.
The proofs of 2)--4) are ``computer proofs''.  The survey \cite{Sha13}
gives an overview of a general method for proving combinatorial
properties of automatic sequences.  We will not explain the method in
any great detail here.  The output of the computer prover is a finite
automaton accepting the binary representation of the lengths of the
squares, cubes, palindromes, etc.\ contained in the sequence of
interest.

Figure~\ref{cube_lengths} shows the automaton accepting the binary
representations of the lengths of the periods of the cubes present in
the sequence.  It is easy to see that the only numbers accepted by the
automaton are $0$ and $3$.  Of course $0$ is not a valid length for
the period of a repetition, but it makes things a little easier
to allow the automaton to accept $0$.

Figure~\ref{square_lengths} shows the automaton accepting the lengths
of the periods of the squares.  Again, it is easy to see from the
structure of the automaton that the non-zero lengths accepted are the
elements of the set $\{1,2\} \cup \{3\cdot2^k : k \geq 0\}$.

Finally, Figure~\ref{pal_lengths} shows the automaton accepting the
lengths of the palindromes.  It is easy to see that this automaton
accepts the binary representations of infinitely many numbers.
\end{proof}

\begin{figure}[p]
\centering
\includegraphics[scale=0.7]{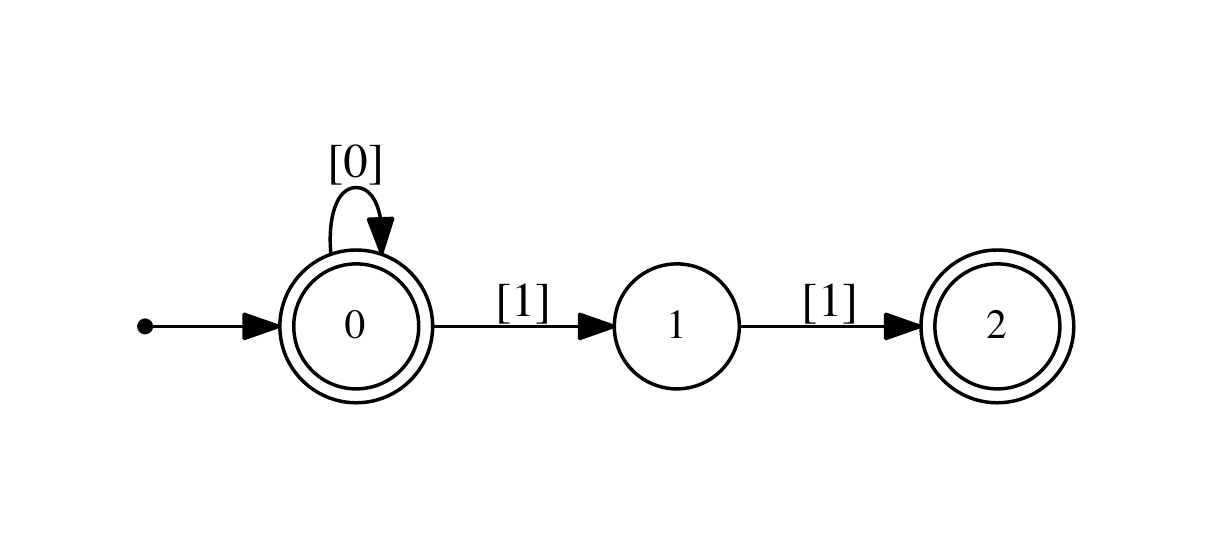}
\caption{Automaton accepting period lengths of cubes in $(i_n)_{n \geq
    0}$}\label{cube_lengths}
\end{figure}

\begin{figure}[p]
\centering
\includegraphics[scale=0.7]{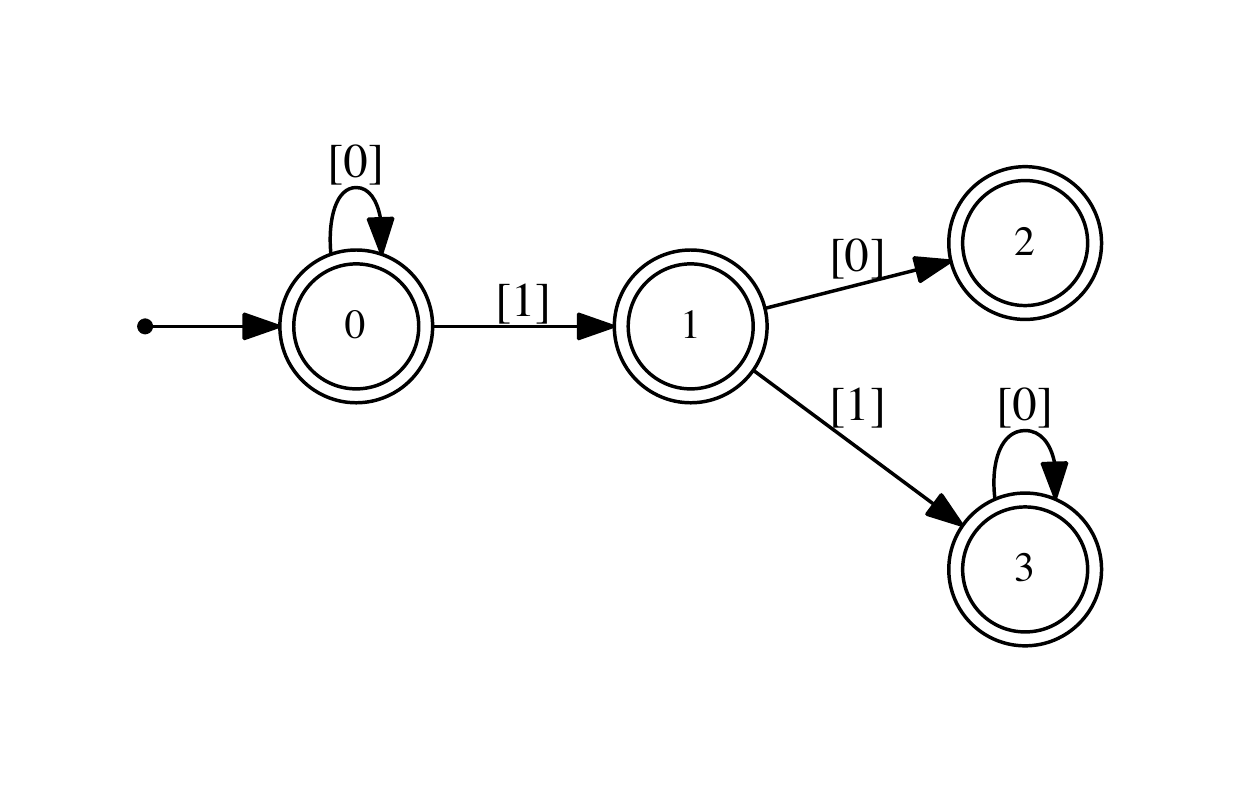}
\caption{Automaton accepting period lengths of squares in $(i_n)_{n \geq
    0}$}\label{square_lengths}
\end{figure}

\begin{figure}[p]
\centering
\includegraphics[scale=0.7]{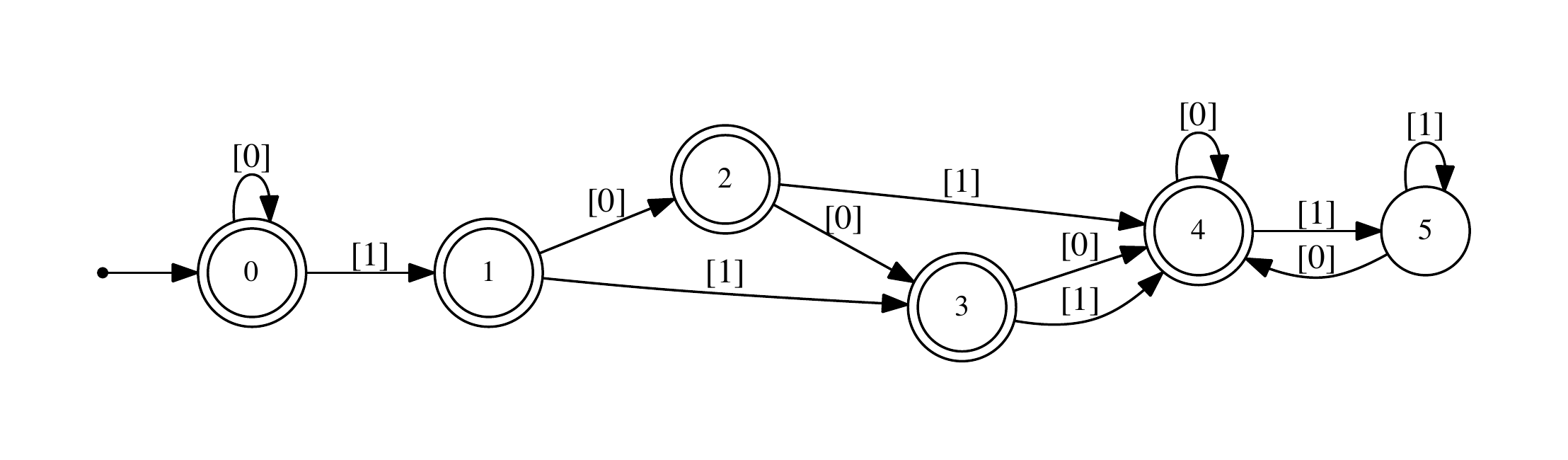}
\caption{Automaton accepting lengths of palindromes in $(i_n)_{n \geq
    0}$}\label{pal_lengths}
\end{figure}

\section{Conclusion}

It would be interesting to study the properties of other sequences of
the form $(\,(-1)^{\sub_{2;w}(n)}\,)_{n \geq 0}$ for different choices
of subsequence $w$.

\section*{Acknowledgments}
The computations needed to prove the results in
Section~\ref{combinat_props} were performed by Jeffrey Shallit and
Hamoon Mousavi.  We would like to sincerely thank them for their
assistance.

\end{document}